\documentclass[11pt,epsf]{article}
\usepackage[dvips]{color}
\newcommand{\bed}{\begin{displaymath}}
\newcommand{\eed}{\end{displaymath}}
\newcommand{\bea}{\bed\begin{array}{rl}}
\newcommand{\eea}{\end{array}\eed}
\newcommand{\barray}{\begin{array}{ll}}
\newcommand{\earray}{\end{array}}

\usepackage{color,latexsym,amsfonts,amssymb,amsmath,amsthm}
\usepackage{colordvi,amsfonts,amsmath,epsfig,subfigure,cite}
\usepackage{graphicx}
\usepackage{dsfont}

\newtheorem{theorem}{Theorem}[section]
\newtheorem{lemma}[theorem]{Lemma}
\newtheorem{remark}{Remark}[section]

\newtheorem{assumption}{Assumption}[section]

\newtheorem{rem}[theorem]{Remark}

\begin{document}

\title{Convergence and stability of two classes of theta-Milstein schemes for stochastic differential equations}

\author{Xiaofeng Zong,\thanks{Institute of Systems Science, Academy of Mathematics and Systems Science, Chinese Academy of Sciences, Beijing 100190,  People¡¯s Republic of China, xfzong87816@gmail.com.  The research of
 this author was supported in part by the National Natural Science Foundations of China (No. 11422110 and No. 61473125) }
\and Fuke Wu, \thanks{School of Mathematics and Statistics, Huazhong University of Science and
 Technology, Wuhan, Hubei 430074, P.R. China, wufuke@mail.hust.edu.cn. The research of
 this author was supported in part by the National Natural Science Foundations of China (No. 11422110 and No. 61473125).}
\and Guiping Xu, \thanks{School of Computer Science and Technology, Huazhong University of Science and Technology, Wuhan, Hubei 430074, P.R. China, gpxu@mail.hust.edu.cn. The research of this author was supported in part by the National Natural Science Foundation of China (Grant No. 61272014).}}
\maketitle

\begin{abstract}
This paper examines convergence and stability of the two classes of theta-Milstein schemes for stochastic differential equations (SDEs) with non-global Lipschitz continuous coefficients: the split-step theta-Milstein (SSTM) scheme and the stochastic theta-Milstein (STM) scheme. For $\theta\in[1/2,1]$, this paper concludes that the two classes of theta-Milstein schemes converge strongly to the exact solution with the order $1$. For $\theta\in[0,1/2]$, under the additional linear growth condition for the drift coefficient, these two classes of the theta-Milstein schemes are also strongly convergent with the standard order. This paper also investigates exponential mean-square stability of these two classes of the theta-Milstein schemes. For $\theta\in(1/2, 1]$, these two theta-Milstein  schemes can share the exponential mean-square stability of the exact solution. For $\theta\in[0, 1/2]$, similar to the convergence, under the additional linear growth condition,  these two theta-Milstein schemes can also reproduce the exponential mean-square stability of the exact solution.

\end{abstract}

\newpage

\section{Introduction}
\setcounter{equation}{0}
During the last decades, SDEs have become increasingly important tools to describe the real world since stochastic models have wide applications in biological systems, neural network, financial engineering and wireless communications. Most SDEs cannot be solved explicitly, so numerical approximations become important tools to study stochastic models. When a numerical scheme is put forward, it is crucial that this numerical scheme can converge to the exact solution. Moreover, it also need to describe the asymptotic properties of the exact solution such as boundedness and stability.

Most of the existing convergence theory for numerical methods of SDEs requires the global Lipschitz condition (see \cite{KP1992, M1995, MT2004}). However, many well-known stochastic systems do not satisfy the global Lipschitz condition, for example, the stochastic Duffing-van der Pol oscillator \cite{AS2007,HJ2012}, stochastic Lorenz equation \cite{S1997Lorenz,HJ2012}, experimental psychology model \cite{HJ2012}, stochastic Ginzburg-Landau equation \cite{KP1992,HJ2012}, stochastic Lotka-Volterra equations \cite{A2003,KP1992,HJ2012} and volatility processes \cite{KP1992,HJ2012} and so on. For some SDEs without the global Lipschitz condition, the classical explicit numerical schemes may not converge to the exact solution in the strong mean-square sense (for example Euler-Maruyama (EM) scheme and Milstein scheme, please see \cite{HJ2009,  HJK2011}). Hence, numerical approximations for SDEs without the global Lipschitz condition have received more and more attention in recent years.

For the SDEs with the one-sided Lipschitz condition (which is weaker than the global Lipschitz condition) on the drift term and the global Lipschitz condition on the diffusion term, Hu \cite{H1996} examined convergence of the backward Euler-Maruyama (BEM) scheme  and obtained that the optimal rate of convergence is $0.5$. Under an additional polynomial condition, Higham et al. \cite{HMS2002,HK2009} and Bastani et al. \cite{BT2012} proved that BEM and split-step BEM schemes converge strongly to the exact solution with the optimal rate $0.5$. Mao and Szpruch \cite{MS2013BEM}  showed that under a dissipative condition on the drift coefficient and the super linear growth condition on the diffusion coefficient, the BEM scheme is convergent with strong order of a half. Under a monotone condition, strong convergence of BEM scheme and theta-EM scheme were investigated by \cite{SM2010, MS2013implicit}. Recently, there are various explicit Euler schemes with one half order were proposed to approximate the SDEs with non-global Lipschitz coefficients, see \cite{MJ2014,MJK2012,S2013,ZWH2014a} and the references therein.  However, it is still an important and difficult work to look for more appropriate conditions to obtain the strong convergence rate of the numerical schemes for SDEs without the global Lipschitz condition.

The explicit Milstein scheme for SDEs developed by Milstein \cite{M1995} can obtain a strong order of convergence higher than Euler-type schemes. Hu et al. \cite{HMY2004} extended this scheme to solve SDEs with time delay. Implicit Milstein schemes under the global Lipschitz condition were studied by Tian and Burrage \cite{TB2001}, Wang et al. \cite{WGW2012}, Alcock \& Burrage \cite{JK2012}. The tamed Milstein scheme was investigated in \cite{WG2013} for SDEs with non-global Lipschitz coefficients.  Recently, Higham et al. \cite{HMS2012} proposed a new Milstein scheme, named as ($\theta,\sigma$)-Milstein scheme, and investigated its strong convergence when SDEs hold polynomial growth for the diffusion term. However, the strong convergence rates of the implicit Milstein schemes for SDEs with non-global Lipschitz coefficients has not yet been well investigated.

Stability of numerical solutions is another central problem for numerical analysis.  The mean-square stability of numerical methods for linear stochastic differential equations have been studied by \cite{H2000a, SM1996, SM2002}. For nonlinear SDEs,  Higham et al.\cite{HMS2003}  showed that the BEM and the split-step BEM can reproduce the mean-square exponential stability of the exact solution. Recently, without the global Lipschitz, Huang \cite{H2012} and authors \cite{ZW1014,ZWH2014} presented conditions under which the stochastic theta method and split-step theta method can not only reproduce the exponential mean-square stability of the exact solution, but also preserve the bound of the Lyapunov exponent for sufficient small stepsize, which measures the decay rate of the numerical solutions. However, there is little work on the mean-square stability of the theta-Milstein schemes for SDEs with non-Lipschitz continuous coefficients.

The main aim is to examine the boundedness and the convergence rate as well as the exponential mean-square stability of theta Milstein schemes for SDEs with non-Lipschitz continuous coefficients. This paper is organized as follows. The next section presents some necessary notations and preliminaries, and then introduces the stochastic theta-Milstein scheme and the split-step theta-Milstein scheme. Section 3 establishes the uniform boundedness of the $p$th moments of the theta-Milstein schemes and shows that for $\theta\in[1/2,1]$, these two classes of theta-Milstein schemes are bounded in the sense of moment, but for $\theta\in[0,1/2]$, the linear growth condition on drift is added to obtain the moment boundedness of the theta-Milstein schemes. Section 4 proves that the theta-Milstein schemes strongly converge to the exact solution with the order $1$. The final section investigates the exponential mean-square stability of these two classes of theta-Milstein schemes.

\section{Notations and preliminaries}\label{S1}
\setcounter{equation}{0}
Throughout this paper, unless otherwise specified, we use the following notations. Let $|\cdot|$ denote both the Euclidean norm in $\mathbb{R}^n$. If $x$ is a  vector, its transpose is denoted by $x^T$ and the inner product is denoted by $\langle x, y\rangle=x^Ty$ for $x,y\in\mathbb{R}^n$. $a\vee b$ represents $\max\{a,b\}$ and $a\wedge b$ denotes $\min\{a, b\}$.  $\mathbb{N}_+$ represents the positive integer set, namely, $\mathbb{N}_+=\{1,2, 3, \ldots\}$. Let $(\Omega,\mathfrak{F},\mathbb{P})$ be a complete probability space with a filtration $\{\mathfrak{F_{t}}\}_{t\geq0}$ satisfying the usual conditions, that is, it is right continuous and increasing while $\mathfrak{F_{0}}$ contains all $\mathbb{P}$-null sets. Let $(w(t))_{t\geq0}$ be a one-dimensional Brownian motion defined on this probability space.

Let $f, g:\mathbb{R}^n\mapsto \mathbb{R}^{n}$ be Borel measurable functions. Let us consider the $n$-dimensional SDE of the form
\begin{equation}\label{SDE}
dx(t)=f(x(t))dt+g(x(t))dw(t),\quad t>0
\end{equation}
 with initial data $x(0)=x_0\in\mathbb{R}^n$.
Assume that $f$ and $g$ satisfy the following assumption:
\begin{assumption}~\label{onesided assumption}
Assume that the functions $f, g\in C^1$ and there exist constants $\mu\in\mathbb{R}$ and $c>0$ such that for any   $x, y\in \mathbb{R}^n$
\begin{equation}\label{onesidedf}
\langle x-y, f(x)-f(y)\rangle \leq \mu|x-y|^2,
\end{equation}
\begin{equation}\label{lipschitzg}
|g(x)-g(y)|^2 \leq c|x-y|^2.
\end{equation}
\end{assumption}

\begin{rem}
{\rm From conditions \eqref{onesidedf} and \eqref{lipschitzg}, we have the following monotone condition
\begin{equation}\label{monotone condition}
 \langle x, f(x)\rangle\vee|g(x)|^2\leq \alpha+\beta|x|^2, \ x\in\mathbb{R}^n,
\end{equation}
where
$$\alpha:=\frac{1}{2}|f(0)|^2\vee2|g(0)|^2 \ {\rm{and}} \ \beta:=(\mu+\frac{1}{2})\vee2c.$$
It is well-known that existence and uniqueness of the global solution of the SDE \eqref{SDE} can be guaranteed  by the local Lipschitz condition and the monotone condition (or general $LV$ condition) (see \cite{M1997}). Since $f, g\in C^1$ implies that $f$ and $g$ satisfy the local Lipschitz condition, the  monotone condition \eqref{monotone condition} can guarantee the existence and uniqueness of the gloal solution for the SDE \eqref{SDE}. This implies that Assumption \ref{onesided assumption} can guarantee the existence and uniqueness of the global solution of the SDE \eqref{SDE}.}
\end{rem}

According to the monotone condition \eqref{monotone condition}, we can present the following lemma (see Lemma 3.2 in \cite{HMS2002}).
\begin{lemma}\label{lemma exact bound}
Under Assumption \ref{onesided assumption}, for each $p\geq 2$ there is $C=C(p,T)>0$ such that
$$
\mathbb{E}\Big[\sup_{0\leq t\leq T}|x(t)|^p\Big]\leq C(1+\mathbb{E}|x_0|^p).
$$
\end{lemma}

In what follows, for the purpose of simplicity, let $C$ represent a generic positive constant independent of the stepsize $\Delta$, whose value may change with each appearance.

Fixed any time $T>0$ and given a stepsize $\Delta=T/N$ for certain integer $N$, we introduce the split-step theta-Milstein (SSTM) approximation
\begin{equation}\label{SSTM}
\left\{
\begin{array}{ll}
y_k=z_k+\theta f(y_k)\Delta, \\
z_{k+1}=z_k+ f(y_k)\Delta+g(y_k)\Delta w_k+\frac{1}{2}L^1g(y_k)(|\Delta w_k|^2-\Delta), \quad k=0, 1,2,\ldots,N.
\end{array}
\right.
\end{equation}
where $y_0=x(0)$, $z_0=y_0-\theta f(y_0)\Delta$, $\theta \in [0,1]$, $L^1=g(x)\frac{\partial}{\partial x}$ and $\Delta w_k:=w((k+1)\Delta)-w(k\Delta)$ be the Brownian increment. This scheme $\{z_k\}_{k\geq 0}$ can be considered as an extension from the split-step theta method developed by our previous work \cite{H2012}. It is interesting that the approximation $\{y_k\}_{k\geq0}$ in \eqref{SSTM} is in fact the stochastic theta-Milstein (STM) approximation
\begin{equation}\label{STM}
y_{k+1}=y_k+\theta f(y_{k+1})\Delta+(1-\theta) f(y_k)\Delta+g(y_k)\Delta w_k+\frac{1}{2}L^1g(y_k)(|\Delta w_k|^2-\Delta),
\end{equation}
which can be proved by substituting $z_k=y_k-\theta f(y_k)\Delta$ into the second equation in \eqref{SSTM}. The STM \eqref{STM}, investigated in \cite{H2000b} for linear scalar SDEs, can be consider as the special case $\sigma=0$ of the $(\theta,\sigma)$-Milstein scheme developed by Higham et al.\cite{HMS2012}. When $\theta=0$, $\{z_k\}_{k\geq 0}$ and $\{y_k\}_{k\geq 0}$ are the classical Milstein approximation proposed in Milstein \cite{M1995}. When $\theta=1$, $\{y_k\}_{k\geq 0}$ becomes the drift-implicit Milstein scheme, which was investigated in \cite{NS2012}. Since theta-Milstein schemes \eqref{SSTM} and \eqref{STM} are semi-implicit when $\theta\in(0, 1]$, to guarantee that they are well defined, we restrict the stepsize $\Delta$ satisfying $\theta\mu\Delta<1$. This, together with the one-sided Lipschitz condition \eqref{onesidedf}, the equation
$$
y=z+\theta\Delta f(y)
$$
has unique solution $y=F_{\Delta,\theta}(z)$ for any $z\in\mathbb{R}^n$.
\section{Uniform boundedness of $p$th moments}
\setcounter{equation}{0}
In order to investigate the boundedness of $p$th moments and convergence of the two classes theta-Milstein approximations, we need the following assumption, which is standard for the classical Milstein scheme.
 \begin{assumption}\label{assu further condition}
Assume that the functions $f, g\in C^2$ and there exists a constant  $\sigma$ such that for any $x,y\in\mathbb{R}^n$
\begin{equation}\label{Lg}
|L^1g(x)-L^1g(y)|^2\leq \sigma|x-y|^2.
\end{equation}
\end{assumption}

Then let us investigate the $p$th moment boundedness of the theta-Milstein approximations for $\theta\in[1/2,1]$ and $\theta\in[0,1/2)$, respectively.

\begin{theorem}\label{thbound121}
Let Assumptions \ref{onesided assumption} and \ref{assu further condition} hold, $\theta\in[1/2,1]$ and let $\Delta<\Delta^*=1/(2\theta\beta)$. Then for each $p\geq2$
  \begin{equation}\label{boundz}
   \mathbb{E}\Big[\sup_{k\Delta\in[0, T]}|z_k|^p\Big]\leq C
\end{equation}
and
\begin{equation}\label{boundy}
  \mathbb{E}\Big[\sup_{k\Delta\in[0, T]}|y_k|^p\Big]\leq C.
\end{equation}
\end{theorem}

\begin{proof}
By the second equation in \eqref{SSTM},
  \begin{eqnarray*}
  \nonumber  |z_{k+1}|^2 &=& |z_k|^2+|f(y_k)|^2\Delta ^2+\frac{1}{4}|L^1g(y_k)(|\Delta w_k|^2-\Delta)|^2\\
  &&+2\Delta z_k^Tf(y_k)+\langle g(y_k)\Delta w_k,L^1g(y_k)(|\Delta w_k|^2-\Delta)\rangle\\
 \nonumber  &&+2\langle z_k+f(y_k)\Delta,g(y_k)\Delta w_k+\frac{1}{2}L^1g(y_k)(|\Delta w_k|^2-\Delta)\rangle+|g(y_k)|^2|\Delta w_k|^2\\
\nonumber&=& |z_k|^2+2\Delta y_k^Tf(y_k)+2|g(y_k)|^2|\Delta w_k|^2+(1-2\theta)|f(y_k)|^2\Delta^2\\
&&+\frac{1}{2}|L^1g(y_k)|^2(|\Delta w_k|^2-\Delta)^2+ \frac{2}{\theta}\langle y_k-(1-\theta) z_k,g(y_k)\Delta w_k\rangle\\
&&+\frac{1}{\theta}\langle y_k-(1-\theta) z_k,L^1g(y_k)(|\Delta w_k|^2-\Delta)\rangle,
\end{eqnarray*}
where we used the equation $z_k =y_k-\theta f(y_k)\Delta$.
Note that $\theta\in[1/2, 1]$. By \eqref{monotone condition}, we have
  \begin{eqnarray*}
  \nonumber  |z_{k+1}|^2&\leq& |z_k|^2+2(\alpha+\beta|y_k|^2)\Delta +2|g(y_k)|^2|\Delta w_k|^2+\frac{1}{2}|L^1g(y_k)|^2(|\Delta w_k|^2-\Delta)^2\\
\nonumber&&+ \frac{2}{\theta}\langle y_k,g(y_k)\Delta w_k\rangle-2\frac{1-\theta}{\theta} \langle z_k,g(y_k)\Delta w_k\rangle+\frac{1}{\theta}\langle y_k,L^1g(y_k)(|\Delta w_k|^2-\Delta)\rangle\\
&&-\frac{1-\theta}{\theta} \langle z_k,L^1g(y_k)(|\Delta w_k|^2-\Delta)\rangle,
\end{eqnarray*}
which implies
  \begin{eqnarray*}
  \nonumber  |z_{k+1}|^2&\leq& |z_0|^2+2\alpha T+\beta\Delta\sum_{i=0}^k|y_i|^2 +2\sum_{i=0}^k|g(y_i)|^2|\Delta w_i|^2\\
\nonumber&&+\frac{1}{\theta}\sum_{i=0}^k\langle y_i,L^1g(y_i)(|\Delta w_i|^2-\Delta)\rangle+ \frac{2}{\theta}\sum_{i=0}^k\langle y_i,g(y_i)\Delta w_i\rangle\\
&&-2\frac{1-\theta}{\theta} \sum_{i=0}^k\langle z_i,g(y_i)\Delta w_i\rangle+\frac{1}{2}\sum_{i=0}^k|L^1g(y_i)|^2(|\Delta w_i|^2-\Delta)^2\\
&&-\frac{1-\theta}{\theta} \sum_{i=0}^k\langle z_i,L^1g(y_i)(|\Delta w_i|^2-\Delta)\rangle.
\end{eqnarray*}
Recall the elementary inequality: for $x_1, \cdots, x_l\geq0$, $p\geq1$, $l=1,2, ..., N$,
$$
\Big(\sum_{i=1}^lx_i\Big)^{p}\leq l^{p-1}\sum_{i=1}^l x_i^{p}.
$$
We therefore have
\begin{eqnarray}\label{mainap}
  \nonumber \frac{1}{8^{p-1}}|z_{k+1}|^{2p} &\leq& (|z_0|^2+2\alpha T)^p+\beta^p\Delta^p\Big(\sum_{i=0}^k|y_i|^2 \Big)^p +2^p\Big(\sum_{i=0}^k|g(y_i)|^2|\Delta w_i|^2\Big)^p\\
  \nonumber &&+2^p\Big|\sum_{i=0}^k\langle y_i,L^1g(y_i)(|\Delta w_i|^2-\Delta)\rangle\Big|^p+ 4^p\Big|\sum_{i=0}^k\langle y_i,g(y_i)\Delta w_i\rangle\Big|^p\\
\nonumber&&+2^p\Big|\sum_{i=0}^k\langle z_i,g(y_i)\Delta w_i\rangle\Big|^p+2^{-p}\Big(\sum_{i=0}^k|L^1g(y_i)|^2(|\Delta w_i|^2-\Delta)^2\Big)^p\\
&&+\Big| \sum_{i=0}^k\langle z_i,L^1g(y_i)(|\Delta w_i|^2-\Delta)\rangle\Big|^p.
\end{eqnarray}
Note that $y_i$ is $\mathfrak{F}_i$- measurable while $\Delta w_i$ is independent of $\mathfrak{F}_i$. Hence, for any integer $m<N$,
\begin{eqnarray}\label{eq2}
  \nonumber \mathbb{E}\Big(\sup_{0\leq k\leq m}\sum_{i=0}^k|g(y_i)|^2|\Delta w_i|^2\Big)^p&\leq&N^{p-1}\mathbb{E}\Big[\sup_{0\leq k\leq m}\sum_{i=0}^k|g(y_i)\Delta w_i|^{2p}\Big]\\
  \nonumber &=& N^{p-1}\mathbb{E}\Big[\sum_{i=0}^m|g(y_i)\Delta w_i|^{2p}\Big]\\
  \nonumber &=& N^{p-1}\sum_{i=0}^m\mathbb{E}|g(y_i)|^{2p}\mathbb{E}|\Delta w_i|^{2p}\\
  \nonumber &\leq& C\Delta\sum_{i=0}^m\mathbb{E}[\alpha+\beta|y_i|^{2}]^p\\
 \nonumber &\leq& C\Delta\sum_{i=0}^m\mathbb{E}[\alpha^p+\beta^p|y_i|^{2p}]\\
   &\leq& C+C\Delta\sum_{i=0}^m\mathbb{E}[|y_i|^{2p}],
\end{eqnarray}
where we also used \eqref{monotone condition}. Using the Burkholder-Davis-Gundy inequality, we have
\begin{eqnarray}\label{eq3}
 \nonumber  \mathbb{E}\Big[\sup_{0\leq k\leq m}\Big|\sum_{i=0}^ky_i^TL^1g(y_i)(|\Delta w_i|^2-\Delta)\rangle\Big|^p\Big] &\leq& C \mathbb{E}\Big[\sum_{i=0}^m|y_i|^2|L^1g(y_i)|^2\Delta^2\Big]^{p/2}\\
 \nonumber  &\leq& C\Delta^p(m+1)^{p/2-1}\sum_{i=0}^m\mathbb{E}[|y_i|^2|L^1g(0)|^2+\sigma|y_i|^{2}]^{p/2}\\
   &\leq& C\Delta^{p/2}+C\Delta^{p/2+1}\sum_{i=0}^m\mathbb{E}[|y_i|^{2p}],
\end{eqnarray}
\begin{eqnarray}\label{eq4}
  \nonumber \mathbb{E}\Big[\sup_{0\leq k\leq m}\Big|\sum_{i=0}^ky_i^Tg(y_i)\Delta w_i\Big|^p\Big] &\leq& C \mathbb{E}\Big[\sum_{i=0}^m|y_i|^2|g(y_i)|^2\Delta\Big]^{p/2}\\
  \nonumber &\leq& C\Delta^{p/2}(m+1)^{p/2-1}\mathbb{E}\sum_{i=0}^m|y_i|^p[\alpha+\beta|y_i|^{2}]^{p/2}\\
 &\leq&C+ C\Delta\sum_{i=0}^m[1+\mathbb{E}[|y_i|^{2p}],
\end{eqnarray}
\begin{eqnarray}\label{eq5}
  \nonumber \mathbb{E}\Big[\sup_{0\leq k\leq m}\Big|\sum_{i=0}^kz_i^Tg(y_i)\Delta w_i\Big|^p\Big] &\leq& C \mathbb{E}\Big[\sum_{i=0}^m|z_i|^2|g(y_i)|^2\Delta\Big]^{p/2}\\
  \nonumber &\leq& C\Delta^{p/2}(m+1)^{p/2-1}\mathbb{E}\sum_{i=0}^m|z_i|^p[\alpha+\beta|y_i|^{2}]^{p/2}\\
 &\leq& C\Delta\sum_{i=0}^m[1+\mathbb{E}[|y_i|^{2p}]+C\Delta\sum_{i=0}^m\mathbb{E}[|z_i|^{2p}],
\end{eqnarray}
and
\begin{eqnarray}\label{eq6}
 \nonumber  \mathbb{E}\Big[\sup_{0\leq k\leq m}\Big|\sum_{i=0}^kz_i^TL^1g(y_i)(|\Delta w_i|^2-\Delta)\rangle\Big|^p\Big] &\leq& C \mathbb{E}\Big[\sum_{i=0}^m|z_i|^2|L^1g(y_i)|^2\Delta^2\Big]^{p/2}\\
 \nonumber  &\leq& C\Delta^p(m+1)^{p/2-1}\sum_{i=0}^m\mathbb{E}[|z_i|^2|L^1g(0)|^2+\sigma|y_i|^{2}]^{p/2}\\
   &\leq& C+C\Delta\sum_{i=0}^m\mathbb{E}[|y_i|^{2p}]+C\Delta\sum_{i=0}^m\mathbb{E}[|z_i|^{2p}].
\end{eqnarray}
Similarly, we have
\begin{eqnarray}\label{eq7}
  \nonumber\mathbb{E}\Big(\sup_{0\leq k\leq m}\sum_{i=0}^k|L^1g(y_i)|^2(|\Delta w_i|^2-\Delta)^2\Big)^p&\leq&  N^{p-1}\mathbb{E}\Big[\sup_{0\leq k\leq m}\sum_{i=0}^k|L^1g(y_i)(|\Delta w_i|^2-\Delta)|^{2p}\Big] \\
  \nonumber &=&  N^{p-1}\mathbb{E}\Big[\sum_{i=0}^m|L^1g(y_i)(|\Delta w_i|^2-\Delta)|^{2p}\Big]\\
  \nonumber &\leq&N^{p-1}\sum_{i=0}^m\mathbb{E}|L^1g(y_i)|^{2p}\mathbb{E}|(|\Delta w_i|^2-\Delta)|^{2p}\\
  \nonumber &=& C\Delta\sum_{i=0}^m\mathbb{E}[|L^1g(0)|^2+\sigma|y_i|^{2}]^p\\
  \nonumber&\leq& C\Delta\sum_{i=0}^m\mathbb{E}[|L^1g(0)|^{2p}+\sigma|y_i|^{2p}]\\
   &\leq& C+C\Delta\sum_{i=0}^m\mathbb{E}[|y_i|^{2p}].
\end{eqnarray}
Combining  \eqref{eq2}-\eqref{eq7} with \eqref{mainap}  yields
\begin{eqnarray}\label{mainap1}
  \nonumber\mathbb{E}\Big[\sup_{0\leq k\leq m+1}|z_{k}|^{2p} \Big]&\leq& C+C\Delta\sum_{i=0}^m\mathbb{E}[|y_i|^{2p}]+C\Delta\sum_{i=0}^m\mathbb{E}[|y_i|^{2p}]\\
  &\leq& C+C\Delta\sum_{i=0}^m\mathbb{E}\Big[\sup_{0\leq k\leq i}|y_k|^{2p}\Big]+C\Delta\sum_{i=0}^m\mathbb{E}\Big[\sup_{0\leq k\leq i}|z_k|^{2p}\Big].
\end{eqnarray}
Using $z_k=y_k-\theta f(y_k)\Delta$ and \eqref{monotone condition}, we have
\begin{eqnarray*}
  |z_k|^2 &=& |y_k|^2-2\theta y_k^Tf(y_k)\Delta+\theta^2\Delta^2 |f(y_k)|^2 \\
  &\geq&  (1-2\theta\Delta\beta)|y_k|^2-2\theta\Delta \alpha,
\end{eqnarray*}
that is,
\begin{eqnarray}\label{1222}
 (1-2\theta\Delta\beta)|y_k|^2\leq |z_k|^2 +2\theta\Delta \alpha,
\end{eqnarray}
which together with \eqref{mainap1} implies that for $\Delta<\Delta^*$
\begin{eqnarray}\label{mainap2}
  \nonumber\mathbb{E}\Big[\sup_{0\leq k\leq m+1}|z_{k}|^{2p} \Big]\leq C+C\Delta\sum_{i=0}^m\mathbb{E}\Big[\sup_{0\leq k\leq i}|z_k|^{2p}\Big].
\end{eqnarray}
Using the discrete-type Gronwall inequality and noting that $(m+1)\Delta\leq T$ give
\begin{eqnarray}\label{mainap3}
 \mathbb{E}\Big[\sup_{0\leq k\leq m+1}|z_{k}|^{2p} \Big]\leq C.
\end{eqnarray}
This together with \eqref{1222} gives the desired assertion \eqref{boundy}.
\end{proof}

The following theorem gives the moment boundedness for $\theta\in[0,1/2)$.
\begin{theorem}\label{th012}
Let Assumptions \ref{onesided assumption} and \ref{assu further condition} hold, $\theta\in[0,1/2)$ and let $\Delta<\Delta_1=1/(2\theta\beta)$ $(\Delta_1=\infty$ if $\theta=0)$. If function $f$ satisfies the linear growth condition
  \begin{equation}\label{f linear growth}
   |f(x)|^2\leq K(1+|x|^2),
  \end{equation}
 then for each $p\geq2$
\begin{equation}\label{boundy1}
  \mathbb{E}\Big[\sup_{k\Delta\in[0, T]}|y_k|^p\Big]\leq C
\end{equation}
and
\begin{equation}\label{boundz1}
   \mathbb{E}\Big[\sup_{k\Delta\in[0, T]}|z_k|^p\Big]\leq C.
\end{equation}
\end{theorem}

\begin{proof}
By \eqref{SSTM},
\begin{eqnarray}\label{main4}
  \nonumber  |z_{k+1}|^2 &=& |z_k|^2+2\langle z_k,f(y_k)\Delta+g(y_k)\Delta w_k+\frac{1}{2}L^1g(y_k)(|\Delta w_k|^2-\Delta)\rangle\\
  &&+\Big|f(y_k)\Delta+g(y_k)\Delta w_k+\frac{1}{2}L^1g(y_k)(|\Delta w_k|^2-\Delta)\Big|^2.
\end{eqnarray}
 Note that
 \begin{equation}\label{z=y}
   z_k=y_k-\theta f(y_k)\Delta.
 \end{equation}
 Substituting this into \eqref{main4} produces
 \begin{eqnarray*}
  \nonumber |z_{k+1}|^2 &\leq& |z_k|^2+2\langle y_k,f(y_k)\Delta+g(y_k)\Delta w_k+\frac{1}{2}L^1g(y_k)(|\Delta w_k|^2-\Delta)\rangle\\
 \nonumber &&+\theta^2\Delta^2|f(y_k)|^2+2|f(y_k)\Delta+g(y_k)\Delta w_k+\frac{1}{2}L^1g(y_k)(|\Delta w_k|^2-\Delta)|^2\\
 \nonumber &\leq&|z_k|^2+2y_k^Tf(y_k)\Delta+2y_k^Tg(y_k)\Delta w_k+y_k^TL^1g(y_k)(|\Delta w_k|^2-\Delta)\rangle\\
 \nonumber &&+7\Delta^2|f(y_k)|^2+6|g(y_k)\Delta w_k|^2+3|L^1g(y_k)(|\Delta w_k|^2-\Delta)|^2\\
 \nonumber &\leq & |z_k|^2+(2\alpha+7|f(0)|^2\Delta)\Delta+(2\beta+7K\Delta)|y_k|^2\Delta\\
 \nonumber &&+2y_k^Tg(y_k)\Delta w_k+y_k^TL^1g(y_k)(|\Delta w_k|^2-\Delta)\rangle\\
  &&+6|g(y_k)\Delta w_k|^2+3|L^1g(y_k)(|\Delta w_k|^2-\Delta)|^2,
\end{eqnarray*}
which implies
\begin{eqnarray}\label{mainoo}
  \nonumber |z_{k+1}|^2 &\leq & |z_0|^2+(2\alpha+7|f(0)|^2\Delta)T+(2\beta+7K\Delta)\Delta\sum_{i=0}^k|y_i|^2\\
 \nonumber &&+2\sum_{i=0}^ky_i^Tg(y_i)\Delta w_i+\sum_{i=0}^ky_i^TL^1g(y_i)(|\Delta w_i|^2-\Delta)\rangle\\
  &&+6\sum_{i=0}^k|g(y_i)\Delta w_i|^2+3\sum_{i=0}^k|L^1g(y_i)(|\Delta w_i|^2-\Delta)|^2.
\end{eqnarray}
Similar to \eqref{mainap}, we have
\begin{eqnarray}\label{mainp}
  \nonumber \frac{1}{6^{p-1}}|z_{k+1}|^{2p} &\leq & (|z_0|^2+(2\alpha+7|f(0)|^2\Delta)T)^p+(2\beta+7K\Delta)^p\Delta^p\Big(\sum_{i=0}^k|y_i|^2\Big)^p\\
 \nonumber &&+2^p\Big|\sum_{i=0}^ky_i^Tg(y_i)\Delta w_i\Big|^p+\Big|\sum_{i=0}^ky_i^TL^1g(y_i)(|\Delta w_i|^2-\Delta)\rangle\Big|^p\\
 &&+6^p\Big(\sum_{i=0}^k|g(y_i)\Delta w_i|^2\Big)^p+3^p\Big(\sum_{i=0}^k|L^1g(y_i)(|\Delta w_i|^2-\Delta)|^2\Big)^p.
\end{eqnarray}
Combining \eqref{eq2}-\eqref{eq4}, \eqref{eq7} with \eqref{mainp} yields
\begin{eqnarray}\label{mainp1}
  \nonumber\mathbb{E}\Big[\sup_{0\leq k\leq m+1}|z_{k}|^{2p} \Big]\leq C+C\Delta\sum_{i=0}^m\mathbb{E}[|y_i|^{2p}],
\end{eqnarray}
which together with \eqref{1222} implies
\begin{eqnarray}\label{mainp2}
  \nonumber\mathbb{E}\Big[\sup_{0\leq k\leq m+1}|z_{k}|^{2p} \Big]\leq C+C\Delta\sum_{i=0}^m\mathbb{E}[|y_i|^{2p}]\leq C+C\Delta\sum_{i=0}^m\mathbb{E}[\sup_{0\leq k\leq i}|z_k|^{2p}].
\end{eqnarray}
Note that $(m+1)\Delta\leq T$. Using the discrete-type Gronwall inequality gives the desired assertion \eqref{boundz1}.
Similarly, using \eqref{1222} and \eqref{boundz1} gives another desired assertion \eqref{boundy1}.
\end{proof}

\begin{remark}
{\rm For $\theta\in[0,1/2)$, the linear growth condition \eqref{f linear growth} may be necessary for guaranteeing the moment boundedness of the theta-Milstein schemes. This fact can be found in \cite[Lemma 3.2]{ZW1014} and \cite{HJ2009,HJK2011}.}
\end{remark}

\section{Convergence of the theta-Milstein approximations}
\setcounter{equation}{0}
Let us now introduce appropriate continuous-time interpolations corresponding to the discrete numerical approximations. More accurately, let us define the continuous solutions $z(t)$ and $y(t)$ for $t\in[t_k,t_{k+1})$ as follows
 \begin{equation}\label{continuous zk}
  \left\{
    \begin{array}{ll}
      z(t)=\displaystyle z(t_k)+(t-t_k)f(y_k)+g(y_k)(w(t)-w(t_k))+\frac{1}{2}L^1g(y_k)(|w(t)-w(t_k)|^2-(t-t_k)),\\
      y(t)=F_{\Delta,\theta}(z(t)),
    \end{array}
  \right.
  \end{equation}
with $z(0)=z_0=y_0-\theta\Delta f(y_0)$, where $t_k=k\Delta$. Note that $z(t)=y(t)-\theta f(y(t))\Delta$. Then the continuous-time approximations $z(t)$ and $y(t)$ are $\mathfrak{F}_t$- measurable. In our analysis it will be more natural to work with the equivalent definition of $z(t)$
\begin{equation}\label{continuous integral zk}
z(t)=z(0)+\int_0^tf(y(\check{s}))ds+\int_0^t g(y(\check{s}))dw(s)+\int_0^t\int_{\check{s}}^sL^1g(y(\check{u}))dw(u)dw(s),
\end{equation}
where $\check{s}=t_k$ for $s\in[t_k,t_{k+1})$. Note that $z(t_k)=z_k$ and $y(t_k)=y_k$. We refer to $z(t)$ and $y(t)$ as the continuous-time extensions of the discrete approximations $z_k$ and $y_k$, respectively.

In what follows, we also use Taylor's formula frequently. If a function $h:\mathbb{R}^n\rightarrow\mathbb{R}^n$ is twice differentiable, the following Taylor's formula holds
\begin{equation}\label{taylor1}
  h(z(s))-h(z(\check{s}))=h'(z(\breve{s}))(z(s)-z(\check{s}))+\bar{R}_s(h),
\end{equation}
where $\bar{R}_s(h)$ is the remainder term defined by
\begin{equation}\label{barR}
  \bar{R}_s(h)=\int_0^1(1-r)h''(z(\check{s})+r(z(s)-z(\check{s})))(z(s)-z(\check{s}),z(s)-z(\check{s}))dr.
\end{equation}
Here for any $u,v\in\mathbb{R}^n$ the derivatives have the following expression
$$
h'(\cdot)(u)=\sum_{i=1}^n\frac{\partial h}{\partial x^i} u^i, \ h''(\cdot)(u,v)=\sum_{i,j=1}^n\frac{\partial^2 h}{\partial x^i\partial x^j} u^i v^j.
$$
Replacing $z(s)-z(\check{s})$ in \eqref{taylor1} with \eqref{continuous zk} and rearranging lead to
\begin{equation}\label{taylor2}
  h(z(s))-h(z(\check{s}))=h'(z(\breve{s}))\Big(\int_{\check{s}}^sg(y(\check{s}))dw(s)\Big)+R_s(h),
\end{equation}
where
$$
R_s(h)=h'(z(\breve{s}))\Big((s-\check{s})f(y(\check{s}))+\frac{1}{2}L^1g(y(\check{s}))(|w(s)-w(\check{s})|^2-(s-\check{s}))\Big)+\bar{R}_s(h).
$$

In order to obtain the strong convergent rate, we also need the following assumption.
\begin{assumption}\label{assu daoshu}
  There exist positive constants $D$ and $q$ such that for all $x\in\mathbb{R}^n$
  \begin{equation}\label{f12}
    |f'(x)|\vee|f''(x)|\leq D(1+|x|^q)
  \end{equation}
  and
   \begin{equation}\label{g2}
     |g''(x)|\leq D.
   \end{equation}
\end{assumption}

Let us firstly give the convergence theorem of the theta-Milstein schemes for $\theta\in[1/2,1]$ as below.
\begin{theorem}\label{th21}
Let Assumptions \ref{onesided assumption}, \ref{assu further condition} and \ref{assu daoshu} hold, $\theta\in(1/2,1]$ and let $\Delta<\Delta^*=1/(2\theta\beta)$. Then for any $p\geq 2$,
  \begin{equation}\label{convergencesst}
   \mathbb{E}\Big[\sup_{t\in[0,T]}|x(t)-z(t)|^p\Big]\leq C\Delta^p
  \end{equation}
  and
  \begin{equation}\label{convergencestm}
   \mathbb{E}\Big[\sup_{t\in[0,T]}|x(t)-y(t)|^p\Big]\leq C\Delta^p.
  \end{equation}
\end{theorem}

To prove Theorem \ref{th21}, we need the following lemmas.
\begin{lemma}\label{lemmaboundfgfnumer}
  Let the conditions in Theorem \ref{th21} hold. Then for any $p\geq2$,
  \begin{equation}\label{eqboundfgf}
 \sup_{0\leq k\leq N}\Big[\mathbb{E}|f(y_k)|^p\Big]\vee\sup_{0\leq k\leq N}\Big[\mathbb{E}|f'(y_k)|^p\Big]\vee\sup_{0\leq k\leq N}\Big[\mathbb{E}|f''(y_k)|^p\Big]\leq C
  \end{equation}
  and
  \begin{equation}\label{boundegg}
  \sup_{0\leq k\leq N}\Big[\mathbb{E}|g(y_k)|^p\Big]\vee\sup_{0\leq k\leq N}\Big[\mathbb{E}|L^1g(y_k)|^p\Big]\leq C.
  \end{equation}
  Moreover, replacing $y_k$ by $z_k$, \eqref{eqboundfgf} and \eqref{boundegg} are also true.
\end{lemma}
\begin{proof}
By the polynomial growth condition \eqref{f12}, the global Lipschitz condition for $g$ and $L^1g$  and Theorem \ref{thbound121} give the desired assertions immediately.
\end{proof}

 Similarly, we can use the polynomial growth condition \eqref{f12}, the global Lipschitz condition for $g$ and $L^1g$  and Lemma \ref{lemma exact bound} to obtain the following lemma, whose proof is omitted.
\begin{lemma}\label{lemmaboundfgfexact}
  Let conditions in Theorem \ref{th21} hold. Then for any $p\geq2$,
  \begin{equation}\label{eqboundfgfex}
 \sup_{0\leq t\leq T}\Big[\mathbb{E}|f(x(t))|^p\Big]\vee \sup_{0\leq t\leq T}\Big[\mathbb{E}|f'(x(t))|^p\Big]\vee \sup_{0\leq t\leq T}\Big[\mathbb{E}|f''(x(t))|^p\Big]\leq C
  \end{equation}
  and
  \begin{equation}\label{boundeggex}
   \sup_{0\leq t\leq T}\Big[\mathbb{E}|g(y(t))|^p\Big]\vee \sup_{0\leq t\leq T}\Big[\mathbb{E}|L^1g(y(t))|^p\Big]\leq C.
  \end{equation}
\end{lemma}

\begin{lemma}\label{lemma numer bound}
   Let conditions in Theorem \ref{th21} hold. For any $p\geq2$,
   \begin{equation}\label{zt}
     \mathbb{E}\Big[\sup_{t\in[0,T]}|z(t)|^p\Big]\vee\mathbb{E}\Big[\sup_{t\in[0,T]}|y(t)|^p\Big]\leq C
   \end{equation}
   and
   \begin{equation}\label{z-y}
     \mathbb{E}\Big[\sup_{t\in[0,T]}|z(t)-y(t)|^p\Big]\leq C\Delta^p.
   \end{equation}
\end{lemma}
\begin{proof}
   For any $p\geq2$ and $t\in[0,T]$, by \eqref{continuous integral zk},
   \begin{eqnarray}\label{L}
    \nonumber \mathbb{E}\Big(\sup_{s\in[0,t]}|z(s)|^p\Big)&\leq&4^{p-1}\mathbb{E}|z(0)|^p+4^{p-1}\mathbb{E}\Big|\int_0^t|f(y(\check{s}))|ds\Big|^p+4^{p-1}\mathbb{E}\Big(\sup_{s\in[0,t]}\Big|\int_0^s g(y(\check{s}))dw(s)\Big|^p\Big)\\
     &&+4^{p-1}\mathbb{E}\Big(\sup_{s\in[0,t]}\Big|\int_0^s\int_{\check{v}}^vL^1g(y(\check{u}))dw(u)dw(v)\Big|^p\Big).
   \end{eqnarray}
   Using Lemma \ref{lemmaboundfgfnumer}, we have
   \begin{eqnarray}\label{L1}
\mathbb{E}\Big|\int_0^t|f(y(\check{s}))|ds\Big|^p&\leq& C\mathbb{E}\int_0^t|f(y(\check{s}))|^p ds\leq C.
   \end{eqnarray}
Applying the Burkholder-Davis-Gundy inequality and  Lemma \ref{lemmaboundfgfnumer} yield
   \begin{eqnarray}\label{L2}
     \nonumber\mathbb{E}\Big(\sup_{s\in[0,t]}\Big|\int_0^s g(y(\check{s}))dw(s)\Big|^p\Big)&\leq& c_p\mathbb{E}\Big(\int_0^t |g(y(\check{s}))|^2ds\Big)^{p/2}\\
 &\leq&C\int_0^t \mathbb{E}|g(y(\check{s}))|^pds\leq C
   \end{eqnarray}
   and
   \begin{eqnarray}\label{L3}
     \nonumber\mathbb{E}\Big(\sup_{s\in[0,t]}\Big|\int_0^s\int_{\check{v}}^vL^1g(y(\check{u}))dw(u)dw(v)\Big)^p&\leq& c_p\mathbb{E}\Big(\int_0^t\Big|\int_{\check{v}}^vL^1g(y(\check{u}))dw(u)\Big|^2dv\Big)^{p/2}\\
    \nonumber &\leq&  C\Delta^{p/2}\int_0^t\mathbb{E}|L^1g(y(\check{v}))|^pdv\\
   &\leq&C,
   \end{eqnarray}
 where $c_p$ is a constant dependent on $p$. Therefore, \eqref{L1}-\eqref{L3} and \eqref{L} produce
   $$
   \mathbb{E}\Big[\sup_{s\in[0,T]}|z(s)|^p\Big]\leq C,
   $$
   which together with \eqref{1222} gives
   $$
   \mathbb{E}\Big[\sup_{s\in[0,T]}|y(s)|^p\Big]\leq C.
   $$
   Note that $z(t)=y(t)-\theta f(y(t))\Delta$. Hence \eqref{z-y} follows from \eqref{zt}.
\end{proof}
\begin{lemma}\label{lemma regular}
   Let conditions in Theorem \ref{th21} hold. For any $p\geq2$ and $t\in[0,T]$,
    \begin{equation}\label{x-xx}
     \mathbb{E}|x(t)-x(\check{t})|^p\leq C\Delta^{p/2}
   \end{equation}
   and
  \begin{equation}\label{z-zx}
     \mathbb{E}|z(t)-z(\check{t})|^p\leq C\Delta^{p/2}.
   \end{equation}
\end{lemma}
\begin{proof}
 By \eqref{SDE} and the Burkholder-Davis-Gundy inequality, we have
 \begin{eqnarray*}
    \mathbb{E}|x(t)-x(\check{t})|^p&=&C \mathbb{E}|\int_{\check{t}}^tf(x(s))ds|^p+C \mathbb{E}|\int_{\check{t}}^tg(x(s))dw(s)|^p\\
    &\leq&C\Delta^{p-1}\int_{\check{t}}^t\mathbb{E}|f(x(s))|^p ds+C \mathbb{E}\Big|\int_{\check{t}}^t|g(x(s))|^2ds\Big|^{p/2}\leq C\Delta^{p/2},
  \end{eqnarray*}
  where we also used Lemma \ref{lemmaboundfgfexact}.
  By \eqref{continuous zk}, we have
  \begin{eqnarray*}
    \mathbb{E}|z(t)-z(\check{t})|^p&=&C\Delta^p \mathbb{E}|f(y(\check{t}))|^p+ C\mathbb{E}|g(y(\check{t}))(w(t)-w(\check{t}))|^p\\
    &&+C\mathbb{E}|L^1g(y(\check{t}))(|w(t)-w(\check{t})|^2-(t-\check{t}))|^p\\
    &\leq&C\Delta^p \mathbb{E}|f(y(\check{t}))|^p+ C\Delta^{p/2}\mathbb{E}|g(y(\check{t}))|^p+C\Delta^p\mathbb{E}|L^1g(y(\check{t}))|^p.
  \end{eqnarray*}
  Then we obtain \eqref{z-zx} from Lemma \ref{lemmaboundfgfnumer}.
\end{proof}

\begin{lemma}\label{lemma rf}
   For any $p\geq2$ and $t\in[0,T]$,
  \begin{equation}\label{rp}
     \mathbb{E}|R_t(f)|^p\leq C\Delta^{p}.
   \end{equation}
\end{lemma}
\begin{proof}
By the definition of $R_t(f)$, we have
\begin{eqnarray*}
  \mathbb{E}|R_t(f)|^p &\leq& \Delta^p3^{p-1}\mathbb{E}|f'(z(\check{t}))(f(y(\check{t})))|^p+3^{p-1}\mathbb{E}|\bar{R}_t(f))|^p\\
  &&+\Delta^p3^{p-1}\frac{1}{2}\mathbb{E}|f'(z(\check{t}))(L^1g(y(\check{t})))|^p\mathbb{E}\Big||w(t)-w(\check{t})|^2-(t-\check{t}))\Big|^p.
\end{eqnarray*}
By Lemmas \ref{lemma numer bound} and \ref{lemma regular} as well as Assumption \ref{assu daoshu}, we can obtain $\mathbb{E}|\bar{R}_t(f))|^p\leq C\Delta^p$. By Lemma \ref{lemmaboundfgfnumer}, we can obtain the desired assertion \eqref{rp}.
\end{proof}

 \begin{proof}[Proof of Theorem \ref{th21}]
  Let $e(t)=x(t)-z(t)$. By \eqref{continuous integral zk} and \eqref{SDE}, we have
\begin{equation}\label{z-x}
e(t)=e(0)+\int_0^t[f(x(s))-f(y(\check{s}))]ds+\int_0^t[g(x(s))-g(y(\check{s}))-\int_{\check{s}}^sL^1g(y(\check{u}))dw(u)]dw(s).
\end{equation}
For the purpose of simplification, define $G(s)=g(x(s))-g(y(\check{s}))-\int_{\check{s}}^sL^1g(y(\check{u}))dw(u)$. Applying It\^{o}'s formula to $|e(t)|^2$ gives
  \begin{eqnarray}\label{z-x2}
      \nonumber |e(t)|^2&=&|e(0)|^2+\int_0^t|G(s)|^2ds+2\int_0^t\langle e(s),G(s)\rangle dw(s)+2\int_0^t\langle e(s),f(x(s))-f(y(\check{s}))\rangle ds\\
     \nonumber &\leq&\theta^2\Delta^2|f(x(0))|^2+\int_0^t|G(s)|^2ds+2\int_0^t\langle e(s),G(s)\rangle dw(s)+2\mu\int_0^t|e(s)|^2 ds\\
     \nonumber &&+2\int_0^t\langle e(s),f(z(\check{s}))-f(y(\check{s}))\rangle ds+2\int_0^t\langle e(s),f(z(s))-f(z(\check{s}))\rangle ds\\
      &=&:\theta^2\Delta^2|f(x(0))|^2+J_1(t)+J_2(t)+2\mu\int_0^t|e(s)|^2 ds+J_3(t)+J_4(t).
    \end{eqnarray}
Applying It\^{o}'s formula for $g(x(s))$ gives
    \begin{eqnarray}\label{gito}
      g(x(s)) = g(x(\check{s}))+\int_{\check{s}}^sL^1g(x(u))dw(u)+\int_{\check{s}}^s(g\diamond f)(x(u)))du,
    \end{eqnarray}
    where $(g\diamond f)(x(u)))=g'(x(u))f(x(u))+\frac{1}{2}trace\{g(x(u))^Tg''(x(u))g(x(u))\}$.
    Substituting \eqref{gito} into $G(s)$ and using  the Burkholder-Davis-Gundy inequality and H\"{o}lder's inequality yield that for any $p\geq2$,
    \begin{eqnarray*}
      \mathbb{E}|G(s)|^p&\leq& 4^{p-1} \mathbb{E}|g(x(\check{s}))-g(y(\check{s}))|^p+4^{p-1}\mathbb{E}\Big|\int_{\check{s}}^s[L^1g(x(u))-L^1g(x(\check{u}))]dw(u)\Big|^p\\
      &&+4^{p-1}\mathbb{E}\Big|\int_{\check{s}}^s[L^1g(x(\check{u}))-L^1g(y(\check{u}))]dw(u)\Big|^p+4^{p-1}\mathbb{E}\Big|\int_{\check{s}}^s(g\diamond f)(x(u)))du\Big|^p\\
      &\leq&C\mathbb{E}|x(\check{s})-y(\check{s})|^p+C\mathbb{E}\Big|\int_{\check{s}}^s|x(u)-x(\check{u})|^2]ds\Big|^{p/2}\\
      &&+C\mathbb{E}\Big|\int_{\check{s}}^s|x(\check{u})-y(\check{u})|^2ds\Big|^{p/2}+C\Delta^p\\
      &\leq&C\mathbb{E}|x(\check{s})-z(\check{s})|^p+C\mathbb{E}|z(\check{s})-y(\check{s})|^p+C\Delta^{p/2-1}\mathbb{E}\int_{\check{s}}^s|x(u)-x(\check{u})|^pds\\
      &&+C\Delta^{p/2-1}\mathbb{E}\int_{\check{s}}^s|x(\check{u})-z(\check{u})|^pds+C\Delta^{p/2-1}\mathbb{E}\int_{\check{s}}^s|z(\check{u})-y(\check{u})|^pds+C\Delta^p\\
      &\leq&C\mathbb{E}|e(\check{s})|^p+C\Delta^p,
    \end{eqnarray*}
    where we also used Assumptions \ref{onesided assumption} and \ref{assu further condition} as well as Lemmas \ref{lemma exact bound}, \ref{lemmaboundfgfnumer}, \ref{lemma regular}. Hence,
     \begin{equation}\label{j1}
       \mathbb{E}\Big[\sup_{t\in[0,r]}|J_1(t)|^{p/2}\Big]=C\mathbb{E}\int_0^r|G(s)|^pds\leq C\int_0^r\mathbb{E}|e(\check{s})|^pds+C\Delta^p.
    \end{equation}
Using the Burkholder-Davis-Gundy inequality  and the H\"{o}lder inequality yields that for any $p\geq2$ and $r\in[0,T]$
     \begin{eqnarray*}
 \nonumber\mathbb{E}\Big[\sup_{t\in[0,r]}|J_2(t)|^{p/2}\Big]&=&2^{p/2}\mathbb{E}\sup_{t\in[0,r]}\Big|\int_0^t\langle e(s),G(s)\rangle dw(s)\Big|^{p/2}\\
     \nonumber  &\leq&2^{p/2}c_p\mathbb{E}\Big(\int_0^r|\langle e(s),G(s)\rangle|^2ds\Big)^{p/4}\\
     \nonumber  &\leq&2^{p/2}c_p\mathbb{E}\Big(\sup_{s\in[0,r]}|e(s)|^2\int_0^r|G(s)|^2ds\Big)^{p/4}
    \end{eqnarray*}
    where $c_p$ is a constant dependent on $p$. Recalling the fundamental inequality: $2ab\leq \kappa_1a^2+\frac{1}{\kappa_1}b^2$ for any $a, b, \kappa_1>0$ yields
     \begin{eqnarray}\label{j2}
 \nonumber\mathbb{E}\Big[\sup_{t\in[0,r]}|J_2(t)|^{p/2}\Big] &\leq&\kappa_1\mathbb{E}\Big[\sup_{s\in[0,r]}|e(s)|^p\Big]+\frac{2^{p}c_p^2}{\kappa_1}\mathbb{E}\Big(\int_0^r|G(s)|^2ds\Big)^{p/2}\\
       &\leq& \kappa_1\mathbb{E}\Big[\sup_{s\in[0,r]}|e(s)|^p\Big]+C\int_0^r\mathbb{E}|e(\check{s})|^pds+C\Delta^p.
    \end{eqnarray}
 Similarly, using H\"{o}lder's inequality, Assumption \ref{assu daoshu} and Lemma \ref{lemma numer bound}, we have
    \begin{eqnarray}\label{j3}
      \nonumber  \mathbb{E}\Big[\sup_{t\in[0,r]}|J_3(t)|^{p/2}\Big]&\leq&C\mathbb{E}\int_0^t| e(s)|^pds+C\mathbb{E}\int_0^r|f(z(\check{s}))-f(y(\check{s}))|^pds\\
      \nonumber &\leq &C\int_0^r\mathbb{E}|e(s)|^pds+C\int_0^r\sqrt{\mathbb{E}(1+|z(\check{s})|^q+|y(\check{s})|^q)^{2p} \mathbb{E}|z(\check{s})-y(\check{s})|^{2p}}ds\\
       &\leq& C\int_0^r\mathbb{E}|e(s)|^pds+C\Delta^p.
    \end{eqnarray}
Now, let us estimate $J_4$. Using \eqref{taylor2} produces
    \begin{eqnarray}\label{j4}
     \nonumber J_4(t) &=& 2\int_0^t\Big\langle e(s),f'(z(\check{s}))\Big(\int_{\check{s}}^sg(y(\check{s}))dw(u)\Big)+R_s(f)\Big\rangle ds \\
      \nonumber &=& 2\int_0^t\Big\langle e(s),f'(z(\check{s}))\Big(\int_{\check{s}}^sg(y(\check{s}))dw(u)\Big)\Big\rangle ds+2\int_0^t\langle x(s)-z(s),R_s(f)\rangle ds\\
       &\leq&J_{41}(t)+\int_0^t|e(s)|^2ds+\int_0^t|R_s(f)|^2 ds,
    \end{eqnarray}
    where $J_{41}(t)=2\int_0^t\langle e(s),f'(z(\check{s}))(\int_{\check{s}}^sg(y(\check{s}))dw(u))\rangle ds.$ It is easy to deduce from Lemma \ref{lemma rf} that
    \begin{equation}\label{rf}
     \mathbb{E}\Big[\sup_{t\in[0,r]}\Big|\int_0^t|R_s(f)|^2 ds\Big|^{p/2}\Big]\leq C\Delta^p.
    \end{equation}
Noting that $e(s)= e(\check{s})+\int_{\check{s}}^s[f(x(u))-f(y(\check{s})]du+\int_{\check{s}}^sG(u)dw(u)$, we have
    \begin{eqnarray*}
      J_{41}(t)&=& 2\int_0^t\int_{\check{s}}^s\langle e(\check{s}),f'(z(\check{s}))(g(y(\check{s})))\rangle dw(u) ds \\
         &&+2\int_0^t\Big\langle \int_{\check{s}}^sG(u)dw(u),f'(z(\check{s}))\Big(\int_{\check{s}}^sg(y(\check{s}))dw(u)\Big)\Big\rangle ds\\
         &&+  2\int_0^t\Big\langle \int_{\check{s}}^s[f(x(u))-f(y(\check{s}))]du,f'(z(\check{s}))\Big(\int_{\check{s}}^sg(y(\check{s}))dw(v)\Big)\Big\rangle ds\\
         &=:&J_{411}(t)+J_{412}(t)+J_{413}(t).
    \end{eqnarray*}
    For any $t\in[0,T]$, let $k_t=\max\{k>0,t_k<t\}$. Define $\hat{s}:=t_{i}$ for $t_{i-1} <s\leq t_{i}$ and $\hat{s}:=t$ for $t_{k_t} <s\leq t$. Applying Fubini's theorem, we can obtain
    \begin{eqnarray*}
      J_{411}(t)/2&=&\sum_{i=1}^{k_t} \int_{t_{i-1}}^{t_{i}}\int_{t_{i-1}}^s\langle e(t_{i-1}),f'(z(t_{i-1}))(g(y(t_{i-1})))\rangle dw(u) ds\\
      &&+\int_{t_{k_t}}^{t}\int_{t_{k_t}}^s\langle e(t_{k_t}),f'(z(t_{k_t}))(g(y(t_{k_t})))\rangle dw(u) ds\\
      &=&\sum_{i=1}^{k_t} \int_{t_{i-1}}^{t_{i}}\int_{u}^{t_i}\langle e(t_{i-1}),f'(z(t_{i-1}))(g(y(t_{i-1})))\rangle ds dw(u) \\
      &&+\int_{t_{k_t}}^{t}\int_{u}^t\langle e(t_{k_t}),f'(z(t_{k_t}))(g(y(t_{k_t})))\rangle ds dw(u) \\
      &=&\sum_{i=1}^{k_t} \int_{t_{i-1}}^{t_{i}}(t_i-u)\langle e(t_{i-1}),f'(z(t_{i-1}))(g(y(t_{i-1})))\rangle ds dw(u) \\
      &&+\int_{t_{k_t}}^{t}(t-u)\langle e(t_{k_t}),f'(z(t_{k_t}))(g(y(t_{k_t})))\rangle ds dw(u)\\
      &=&\int_{0}^{t}(\hat{u}-u)\langle e(\check{u}),f'(z(\check{u}))(g(y(\check{u})))\rangle dw(u).
    \end{eqnarray*}
For any $\kappa_2>0$, using the Burkholder-Davis-Gundy inequality, H\"{o}lder's inequality and Lemma \ref{lemmaboundfgfnumer} gives
   \begin{eqnarray}\label{j411}
    \nonumber \mathbb{E}\Big[\sup_{t\in[0,r]} |J_{411}(t)|^{p/2}\Big] &\leq &c_p\mathbb{E}\Big(\int_{0}^{r}2|(\hat{u}-u)\langle e(\check{u}),f'(z(\check{u}))(g(y(\check{u})))\rangle|^2 du\Big)^{p/4}\\
    \nonumber &\leq&c_p\mathbb{E}\Big(\int_{0}^{r}2|e(\check{u})|^2\Delta^2|f'(z(\check{u}))(g(y(\check{u})))|^2 du\Big)^{p/4}\\
    \nonumber &\leq&c_p \mathbb{E}\Big(2\sup_{s\in[0,r]}|e(s)|^2\Delta^2\int_{0}^{r}|f'(z(\check{u}))(g(y(\check{u})))|^2 du\Big)^{p/4}\\
     \nonumber &\leq& \kappa_2\mathbb{E}\Big[\sup_{s\in[0,r]}|e(s)|^p\Big]+\frac{c_p^22^{p/2}}{\kappa_2}\Delta^p\mathbb{E}\Big(\int_{0}^{r}|f'(z(\check{u}))(g(y(\check{u})))|^2 du\Big)^{p/2}\\
     &\leq& \kappa_2\mathbb{E}\Big[\sup_{s\in[0,r]}|e(s)|^p\Big]+C\Delta^p,
    \end{eqnarray}
Similarly,
   \begin{eqnarray}\label{j412}
   \nonumber  \mathbb{E}\Big[\sup_{t\in[0,r]} |J_{412}(t)|^{p/2}\Big] &\leq&C\mathbb{E}\int_0^r\Big|\Big\langle \int_{\check{s}}^sG(u)dw(u),f'(z(\check{s}))\Big(\int_{\check{s}}^sg(y(\check{s}))dw(u)\Big)\Big\rangle\Big|^{p/2} ds\\
    \nonumber &\leq&C \int_0^r  \sqrt{\mathbb{E}\Big|\int_{\check{s}}^sG(u)dw(u)\Big|^p\mathbb{E}\Big|\int_{\check{s}}^s f'(z(\check{s}))(g(y(\check{s})))dw(u)\Big|^p}ds\\
   \nonumber  &\leq&C\int_0^r  \sqrt{\mathbb{E}\Big|\int_{\check{s}}^s|G(u)|^2du\Big|^{p/2}\mathbb{E}\Big|\int_{\check{s}}^s |f'(z(\check{s}))(g(y(\check{s})))|^2du\Big|^{p/2}}ds\\
    \nonumber &\leq&C\int_0^r  \sqrt{\Delta^{p-2}\int_{\check{s}}^s\mathbb{E}|G(u)|^pdu\int_{\check{s}}^s \mathbb{E}|f'(z(\check{s}))(g(y(\check{s})))|^pdu}ds\\
    \nonumber &\leq&C\int_0^r  \sqrt{\Delta^{p}[\mathbb{E}|e(\check{s})|^p+\Delta^p]}ds\\
     &\leq&C\int_0^r\mathbb{E}|e(\check{s})|^pds+C\Delta^p .
    \end{eqnarray}
  In order to estimate $J_{413}$, we divide it into the following three parts
  \begin{eqnarray}\label{j413}
\nonumber J_{413}(t)&=&2\int_0^t\Big\langle \int_{\check{s}}^s[f(x(u))-f(x(\check{u}))]du,f'(z(\check{s}))\Big(\int_{\check{s}}^sg(y(\check{s}))dw(v)\Big)\Big\rangle ds\\
\nonumber&&+2\int_0^t\Big\langle \int_{\check{s}}^s[f(z(\check{u}))-f(y(\check{s}))]du,f'(z(\check{s}))\Big(\int_{\check{s}}^sg(y(\check{s}))dw(v)\Big)\Big\rangle ds\\
\nonumber &&+2\int_0^t\Big\langle \int_{\check{s}}^s[f(x(\check{u}))-f(z(\check{u}))]du,f'(z(\check{s}))\Big(\int_{\check{s}}^sg(y(\check{s}))dw(v)\Big)\Big\rangle ds\\
&=&:I_1(t)+I_2(t)+I_3(t).
  \end{eqnarray}
Using the Burkholder-Davis-Gundy inequality and H\"{o}lder's inequality give  that for any $p\geq2$,
\begin{eqnarray}\label{i1}
   \nonumber && \mathbb{E}\Big[\sup_{t\in[0,r]} |I_{1}(t)|^{p/2}\Big]\\
  \nonumber   &&\leq C\mathbb{E}\int_0^r\Big|\Big\langle \int_{\check{s}}^s[f(x(u))-f(x(\check{u}))]du,f'(z(\check{s}))\Big(\int_{\check{s}}^sg(y(\check{s}))dw(u)\Big)\Big\rangle\Big|^{p/2} ds\\
   \nonumber  &&\leq  C \int_0^r  \sqrt{\mathbb{E}\Big|\int_{\check{s}}^s[f(x(u))-f(x(\check{u}))]du\Big|^p\mathbb{E}\Big|\int_{\check{s}}^s f'(z(\check{s}))(g(y(\check{s})))dw(u)\Big|^p}ds\\
   \nonumber  &&\leq C\int_0^r  \sqrt{\Delta^{p-1}\int_{\check{s}}^s\mathbb{E}|f(x(u))-f(x(\check{u}))|^pdu\Delta^{p/2-1}\int_{\check{s}}^s \mathbb{E}|f'(z(\check{s}))(g(y(\check{s})))|^pdu}ds\\
  &&\leq C\int_0^r  \sqrt{\Delta^{p-1}\Delta\Delta^{p/2}\Delta^{p/2-1}\Delta}ds\leq C\Delta^p,
\end{eqnarray}
where we used Lemma \ref{lemmaboundfgfnumer} and \ref{lemma regular}. Similarly,
\begin{eqnarray}\label{i2}
     \nonumber\mathbb{E}\Big[\sup_{t\in[0,r]} |I_{2}(t)|^{p/2}\Big] &\leq&(2T)^{p/2}\mathbb{E}\int_0^r\Big|\Big\langle \int_{\check{s}}^s[f(z(\check{u}))-f(y(\check{u}))]du,f'(z(\check{s}))\Big(\int_{\check{s}}^sg(y(\check{s}))dw(u)\Big)\Big\rangle\Big|^{p/2} ds\\
    \nonumber &\leq&C \int_0^r  \sqrt{\mathbb{E}\Big|\int_{\check{s}}^s[f(z(\check{u}))-f(y(\check{u}))]du\Big|^p\mathbb{E}\Big|\int_{\check{s}}^s f'(z(\check{s}))(g(y(\check{s})))dw(u)\Big|^p}ds\\
  \nonumber   &\leq&C\int_0^r  \sqrt{\Delta^p\mathbb{E}|f(z(\check{s}))-f(y(\check{s}))|^p\Delta^{p/2-1}\int_{\check{s}}^s \mathbb{E}|f'(z(\check{s}))(g(y(\check{s})))|^pdu}ds\\
&\leq&C\int_0^r  \sqrt{\Delta^{p}\Delta^{p}\Delta^{p/2}}ds\leq C\Delta^p .
\end{eqnarray}
Applying Fubini's theorem obtains
\begin{eqnarray*}
I_3(t)&=&2\int_0^t \int_{\check{s}}^s(s-\check{s})\langle f(x(\check{s}))-f(z(\check{s})),f'(z(\check{s}))(g(y(\check{s})))\rangle dw(v) ds\\
&=&2\int_0^t \int_{v}^{\hat{v}}(s-\check{s})\langle f(x(\check{s}))-f(z(\check{s})),f'(z(\check{s}))(g(y(\check{s})))\rangle ds dw(v).
\end{eqnarray*}
Therefore,
\begin{eqnarray}\label{i3}
  \nonumber   \mathbb{E}\Big[\sup_{t\in[0,r]} |I_{3}(t)|^{p/2}\Big] &\leq&C \mathbb{E}\Big(\int_0^r  \Big|\int_{v}^{\hat{v}}(s-\check{v})\langle f(x(\check{v}))-f(z(\check{v})),f'(z(\check{v}))(g(y(\check{v})))\rangle ds\Big|^2 dv\Big)^{p/4}\\
  \nonumber   &\leq&C \mathbb{E}\Big(\int_0^t \Big|\int_{v}^{\hat{v}}(s-\check{v})ds\langle f(x(\check{v}))-f(z(\check{v})),f'(z(\check{v}))(g(y(\check{v})))\rangle\Big|^2 dv\Big)^{p/4}\\
 \nonumber    &\leq&C\Delta^p\int_0^r\mathbb{E}\langle f(x(\check{v}))-f(z(\check{v})),f'(z(\check{v}))(g(y(\check{v})))\rangle\Big|^{p/2}dv\\
     &\leq&C\Delta^p.
\end{eqnarray}
By \eqref{i1}-\eqref{i3}, we obtain from \eqref{j413}
  \begin{equation*}
  \mathbb{E}\Big[\sup_{t\in[0,r]} |J_{413}(t)|^{p/2}\Big]\leq C\Delta^p,
  \end{equation*}
  which together with \eqref{j411} and \eqref{j412} implies
   \begin{eqnarray}\label{j41}
  \nonumber \mathbb{E}\Big[\sup_{t\in[0,r]} |J_{41}(t)|^{p/2}\Big]&\leq& 3^{p/2-1}\Big( \mathbb{E}\Big[\sup_{t\in[0,r]} |J_{411}(t)|^{p/2}\Big]+\mathbb{E}\Big[\sup_{t\in[0,r]} |J_{412}(t)|^{p/2}\Big]\\
\nonumber  &&+\mathbb{E}\Big[\sup_{t\in[0,r]} |J_{413}(t)|^{p/2}\Big]\Big)\\
  &\leq& 3^{p/2-1}\kappa_2\mathbb{E}\Big[\sup_{s\in[0,r]}|e(s)|^p\Big]+C\int_0^r\mathbb{E}|e(\check{s})|^pds+C\Delta^p.
  \end{eqnarray}
  Combining \eqref{j41}, \eqref{rf} with \eqref{j4} gives
   \begin{eqnarray}\label{j40}
     \nonumber   \mathbb{E}\Big[\sup_{t\in[0,r]} |J_4(t)|^{p/2}\Big] &\leq&3^{p-2} \kappa_2\mathbb{E}\Big[\sup_{s\in[0,r]}|e(s)|^p\Big]+C\int_0^r\mathbb{E}|e(\check{s})|^pds+C\Delta^p+C\int_0^r\mathbb{E}|e(s)|^pds.
    \end{eqnarray}
    Let $\kappa=(\kappa_1+3^{p-2} \kappa_2)6^{p/2-1}$. Combining \eqref{j1}, \eqref{j2}, \eqref{j3}, \eqref{j40} and \eqref{z-x2} produces
     \begin{eqnarray*}
     \nonumber   \mathbb{E}\Big[\sup_{t\in[0,r]} |e(t)|^{p}\Big] &\leq& 6^{p/2-1}\sum_{l=1}^4\mathbb{E}\Big[\sup_{t\in[0,r]} |J_l(t)|^{p/2}\Big]+C\Delta^p+C\int_0^r|e(s)|^pds\\
     &\leq& \kappa\mathbb{E}\Big[\sup_{s\in[0,r]}|e(s)|^p\Big]+C\Delta^p+C\int_0^r\mathbb{E}\Big[\sup_{u\in[0,s]}|e(u)|^p\Big]ds.
    \end{eqnarray*}
   Therefore, letting $\kappa_1$ and $\kappa_2$ be sufficiently small such that $\kappa<1$, we have
     \begin{eqnarray}\label{e}
    \mathbb{E}\Big[\sup_{t\in[0,r]} |e(t)|^{p}\Big] \leq C\int_0^r\mathbb{E}\Big[\sup_{u\in[0,s]} |e(s)|^p\Big]ds+C\Delta^p.
    \end{eqnarray}
    Thus, the result \eqref{convergencesst} follows from the Gronwall inequality. The desired assertion \eqref{convergencestm} follows from \eqref{convergencesst} and Lemma \ref{lemma numer bound}.
 \end{proof}
 For $\theta\in[0,1/2]$, we can use the similar techniques to obtain the following convergence theorem. Its proof is omitted.
\begin{theorem}\label{thconv012}
Let Assumptions \ref{onesided assumption}, \ref{assu further condition} and \ref{assu daoshu} hold, $\theta\in[0,1/2)$ and let $\Delta<\Delta_1=1/(2\theta\beta)$ $(\Delta_1=\infty$, if $\theta=0)$. If the function $f$ satisfies the linear growth condition \eqref{f linear growth}, then for any $p\geq 2$ we have
  \begin{equation*}
   \mathbb{E}\Big[\sup_{t\in[0,T]}|x(t)-z(t)|^p\Big]\leq C\Delta^p
  \end{equation*}
  and
  \begin{equation*}
   \mathbb{E}\Big[\sup_{t\in[0,T]}|x(t)-y(t)|^p\Big]\leq C\Delta^p.
  \end{equation*}
\end{theorem}

\begin{remark}
{\rm  Convergence for $\theta\in(1/2,1]$ in Theorem \ref{th21} are partially based on the Assumption \ref{onesided assumption}. This does not imply that condition \eqref{lipschitzg} in Assumption \ref{onesided assumption} is necessary. For example, Kloeden and Neuenkirch \cite{KN2012} showed that the semi-implicit Milstein scheme (STM scheme with $\theta=1$) converges to the exact solution for the Cox-Ingersoll-Ross process, whose diffusion term is non-Lipschitz continuous.}
\end{remark}
%
\begin{remark}
{\rm Higham and his coauthors \cite{HMS2002, HK2009} showed that the split step backward Euler and backward Euler schemes strongly converge to the exact solution with the standard order $0.5$. In fact, by the similar skills used in the proof of Theorem \ref{th21}, we can show that the convergence rate $0.5$ holds for theta-Euler schemes (split-step theta and stochastic theta schemes) with $\theta\in(1/2,1]$. Moreover, for the SODE with additive noise, that is, $g(x)$ is a constant, then $L^1g(x)=0$, the theta-Milstein schemes are degenerated to the theta-Euler schemes. Hence, our convergence results show that the theta-Euler schemes ($\theta\in(1/2,1]$) is convergent with order $1$, which is a new result for the SODE with one-sided Lipschitz and polynomial Lipschitz continuous drift. }
\end{remark}

In this section, although the strong convergence results for the theta-Milstein schemes are established under Assumptions \ref{onesided assumption}, \ref{assu further condition} and \ref{assu daoshu}, the following exponential mean-square stability only depends on the one-sided  Lipschitz condition \eqref{onesidedf} and the local Lipschitz condition.

\section{Stability analysis of the theta-Milstein schemes for SDEs}
\setcounter{equation}{0}
 For the purpose of stability, assume that $f(0)=g(0)=0$. This shows that (\ref{SDE}) admits a trivial solution. Then inequality \eqref{Lg} in Assumption \ref{assu further condition} becomes
 \begin{equation}\label{lgcondition}
  |L^1g(x)|^2\leq \sigma |x|^2.
 \end{equation}
In this section, we assume that $f$ and $g$ satisfy the following local Lipschitz condition, which is classical for the nonlinear SDEs.
\begin{assumption}~(Local Lipschitz condition)\label{local Lipschitz condition}
$f$ and $g$ satisfy the local Lipschitz condition,  that is, for each $j>0$ there exists a positive constant $K_j$ such that for any   $x, y\in \mathbb{R}^n$ with $|x|\vee|y| \leq j$,
\begin{equation}\label{local}
| f(x)-f(y)| \vee |g(x)-g(y)|  \leq K_j|x-y|.
\end{equation}
\end{assumption}

 To investigate stability of numerical approximations, let us firstly give the stability criterion of SDE \eqref{SDE} (also see \cite{M1997}):
\begin{theorem}\label{th exact}
  Let Assumption \ref{local Lipschitz condition} hold. If there exists a positive constants $\gamma$ such that for all $x,y\in\mathbb{R}^{n}$,
  \begin{equation}\label{coupled condition}
   2x^Tf(x)+|g(x)|^2\leq -\gamma |x|^2,
  \end{equation}
  then the solution of \eqref{SDE} has the property
  \begin{equation}\label{exponent}
   \mathbb{E}|x(t)|^2\leq C(x_0)e^{-\gamma t},
  \end{equation}
  where $C(x_0)$ is a positive constant dependent on the initial data $x_0$.
\end{theorem}

Under the one-sided Lipschitz condition \eqref{onesidedf}, we define
 $$
 \Delta^*= \left\{
    \begin{array}{ll}
     \displaystyle \frac{1}{\mu\theta},\ \ \ \ & \theta\in(0,1),\ \mbox{and}\ \mu>0, \\
      \infty, &\theta=0 \ \mbox{or}\ \mu<0.
    \end{array}
  \right.
$$
When $\Delta<\Delta^*$, the two classes of theta-Milstein schemes are well defined. Moreover, we have the following stability result.
\begin{theorem}\label{thSSST}
Let the conditions in Theorem \ref{th exact} hold. Under condition \eqref{lgcondition} and the one-sided  Lipschitz condition \eqref{onesidedf}, the following two stability results hold.
    \begin{description}
     \item[$(i)$]  Let $\theta \in[0,1/2]$. If there exists a positive constant $K$ such that for any $x\in\mathbb{R}^n$,
    \begin{equation}\label{linear growth}
     |f(x)|^2\leq K|x|^2,
    \end{equation}
    then for any $\Delta<\Delta_1^*=\displaystyle\frac{\gamma}{(1-2\theta)K+0.5\sigma}\wedge\Delta^*$, $k\in\mathbb{N}_+$,
      \begin{equation*}
    \mathbb{E}|z_k|^2\leq C(x_0)e^{-\gamma_{\Delta}k\Delta} \quad {\rm{and}} \quad \mathbb{E}|y_k|^2\leq C(x_0)e^{-\gamma_{\Delta}k\Delta},
    \end{equation*}
    where $\displaystyle\gamma_{\Delta}=-\frac{1}{\Delta}\log\Big(1-\frac{\gamma-(1-2\theta)K\Delta-0.5\sigma\Delta}{(1+\theta\Delta\sqrt{K})^2}\Delta\Big).$
    \item[$(ii)$] Let $\theta \in(1/2,1]$. Then for any $\displaystyle\Delta<\Delta_2^*=\frac{2\gamma}{\sigma}\wedge\Delta^*$,  $k\in\mathbb{N}_+$,
    \begin{equation*}
    \mathbb{E}|z_k|^2\leq C(x_0)e^{-\gamma_{\Delta}k\Delta} \quad {\rm{and}} \quad \mathbb{E}|y_k|^2\leq C(x_0)e^{-\gamma_{\Delta}k\Delta},
    \end{equation*}
    where $\displaystyle\gamma_{\Delta}=-\frac{1}{\Delta}\log\Big(1-\frac{(2\theta-1)(\gamma-0.5\sigma\Delta)\Delta}{2\theta-1+(\gamma-0.5\sigma\Delta)\Delta\theta^2}\Big)$.

  \end{description}
\end{theorem}
\begin{proof}
For any large $l>0$, define the stoping time
\begin{equation}\label{c3stoptdef}
\lambda_l=\inf\{i>0: |y_i|>l \ \ {\rm{or}}\ \ |z_i|>l\}.
\end{equation}
It is observed that $|y_{k-1}|\leq l$, $|z_{k-1}|\leq l$ for $k\in[0,\lambda_l]$. By \eqref{SSTM},
$$|z_{k}|^2 \leq  4\Big(|z_{k-1}|^2+|f(y_{k-1})|^2\Delta^2+|g(y_{k-1})|^2|\Delta w_{k-1}|^2+ \frac{1}{4}|L^1g(y_{k-1})|^2||\Delta w_{k-1}|^2-\Delta|^2           \Big).$$
Then we can obtain that
\begin{eqnarray*}
\mathbb{E}[|z_{k}|^2\mathbf{1}_{[0,\lambda_l]}(k)] &\leq& 4l^2+4\mathbb{E}[|f(y_{k-1})|^2\mathbf{1}_{[0,\lambda_l]}(k)]\Delta^2+4\mathbb{E}[|g(y_{k-1})|^2|\Delta w_{k-1}|^2\mathbf{1}_{[0,\lambda_l]}(k)]\\
&&+\mathbb{E}[|L^1g(y_{k-1})|^2||\Delta w_{k-1}|^2-\Delta|^2 \mathbf{1}_{[0,\lambda_l]}(k)]\\
&\leq& 4l^2+4\mathbb{E}[|f(y_{k-1})|^2\mathbf{1}_{[0,\lambda_l]}(k)]\Delta^2+4\Big(\mathbb{E}[|g(y_{k-1})|^4\mathbf{1}_{[0,\lambda_l]}(k)]\mathbb{E}[|\Delta w_{k-1}|^4]\Big)^{1/2}\\
&&+\Big(\mathbb{E}[|L^1g(y_{k-1})|^4 \mathbf{1}_{[0,\lambda_l]}(k)]\mathbb{E}[|(\Delta w_{k-1})^2-\Delta|^4]\Big)^{1/2},
\end{eqnarray*}
where we used H\"{o}lder's inequality. Making use of the property $\mathbb{E}[|\Delta w_k|^{2i}]=(2i-1)!!\Delta^i$, we have  $\mathbb{E}[|\Delta w_{k-1}|^4]=3\Delta^2$ and $\mathbb{E}[|(\Delta w_{k-1})^2-\Delta|^4]\leq 2^3(\mathbb{E}[|\Delta w_{k-1}|^8]+\Delta^4)=2^3(7!!+1)\Delta^4$, where $(2i-1)!!=(2i-1)(2i-3)\cdots 3\cdot1$ for $i=1, 2, \ldots$.
By  Assumption \ref{local Lipschitz condition} and \eqref{lgcondition}, it is easy to deduce that $|f(y_{k-1})|^2\mathbf{1}_{[0,\lambda_l]}(k)\leq K_l^2l^2$, $|g(y_{k-1})|^4\mathbf{1}_{[0,\lambda_l]}(k)\leq K_l^4l^4$ and $|L^1g(y_{k-1})|^4 \mathbf{1}_{[0,\lambda_l]}(k)\leq \sigma^2 l^4$. Therefore,
\begin{equation}\label{524b}
 \mathbb{E}[|z_{k}|^2\mathbf{1}_{[0,\lambda_l]}(k)] \leq 4l^2+4K_l^2l^2\Delta^2+8K_l^2l^2\Delta+4\sqrt{7!!+1}\sigma l^2\Delta^2=:K(l),
\end{equation}
 which together with  $z_k=y_k-\theta \Delta f(y_k,y_{k-m})$ and $ 2x^Tf(x)\leq -\gamma |x|^2$ yields
 \begin{equation}\label{z>yk}
 |z_{k}|^2 =  |y_{k}|^2-2\theta\Delta y_{k}^Tf(y_{k})+|f(y_{k})|^2\Delta^2\theta^2\geq (1+\gamma\theta\Delta)|y_{k}|^2.
 \end{equation}
Hence,
  \begin{equation}\label{aaaa}
  \mathbb{E}[|y_k|^2\mathbf{1}_{[0,\lambda_l]}(k)]\leq  K(l).
  \end{equation}
By \eqref{SSTM},
 \begin{eqnarray*}
  \nonumber  |z_{k+1}|^2 &=& |z_k|^2+|f(y_k)|^2\Delta ^2+\frac{1}{4}|L^1g(y_k)(|\Delta w_k|^2-\Delta)|^2+\langle g(y_k)\Delta w_k,L^1g(y_k)(|\Delta w_k|^2-\Delta)\rangle\\
 \nonumber  &&+2\langle z_k+f(y_k)\Delta,g(y_k)\Delta w_k+\frac{1}{2}L^1g(y_k)(|\Delta w_k|^2-\Delta)\rangle+|g(y_k)|^2|\Delta w_k|^2+2\Delta z_k^Tf(y_k)\\
&=& |z_k|^2+2\Delta y_k^Tf(y_k)+|g(y_k)|^2\Delta+(1-2\theta)|f(y_k)|^2\Delta^2+\frac{1}{2}\Delta^2|L^1g(y_k)|^2+m_k,
\end{eqnarray*}
where $m_k=2\langle y_k+(1-\theta)f(y_k)\Delta,g(y_k)\Delta w_k+\frac{1}{2}L^1g(y_k)(|\Delta w_k|^2-\Delta)\rangle+|g(y_k)|^2(|\Delta w_k|^2-\Delta)+\langle g(y_k)\Delta w_k$, $L^1g(y_k)(|\Delta w_k|^2-\Delta)\rangle+\frac{1}{4}|L^1g(y_k)|^2[(|\Delta w_k|^2-\Delta)^2-2\Delta^2]$.
Using condition \eqref{lgcondition} and the coupled condition \eqref{coupled condition}, we get
 \begin{equation}\label{main2}
  |z_{k+1}|^2 \leq  |z_k|^2-\gamma\Delta |y_k|^2 +(1-2\theta)\Delta^2 |f(y_k)|^2+0.5\Delta^2\sigma|y_k|^2+m_k.
\end{equation}

\ \textbf{Case (i)}\ $\theta\in[0, 1/2)$: \ \
Note that the linear growth condition of $f$ and \eqref{coupled condition} imply
\begin{equation}\label{c2q11}
 |g(x)|^2 \leq -2x^Tf(x)\leq  |x|^2+|f(x)|^2\leq (1+K)|x|^2.
\end{equation}
 From conditions \eqref{lgcondition}, \eqref{linear growth} and \eqref{c2q11}, there is a positive constant $\bar{C}$ such that
\begin{equation}\label{524a}
|m_k|\leq [\bar{C}+\bar{C}|\Delta w_k|^2+\bar{C}|\Delta w_k|^4]|y_k|^2.
\end{equation}
 Noting that $y_k$ and $\mathbf{1}_{[0,\lambda_l]}(k)$ are $\mathfrak{F}_{k\Delta}$- measurable while $\Delta w_k$ is independent of $\mathfrak{F}_{k\Delta}$, then we obtain from \eqref{aaaa} that $\mathbb{E}[ m_k \mathbf{1}_{[0,\lambda_l]}(k)]=0$.  The linear growth condition \eqref{linear growth} gives
\begin{eqnarray}\label{maina}
 |z_{k+1}|^2 \leq |z_k|^2+[(1-2\theta)K\Delta+0.5\sigma\Delta-\gamma]\Delta |y_k|^2+m_k.
\end{eqnarray}
Then for $\Delta<\Delta^*_1$ and sufficiently large $l$, by \eqref{main2} and using linear growth condition \eqref{linear growth}, we have
\begin{eqnarray}\label{c2xyz9}
\nonumber  |z_{k\wedge\lambda_l}|^2  &\leq& |z_0|^2+[(1-2\theta)K\Delta+0.5\sigma\Delta-\gamma]\Delta \sum_{i=0}^{(k\wedge\lambda_l)-1}|y_{i}|^2+\sum_{i=0}^{(k\wedge\lambda_l)-1}m_i\\
&\leq&|z_0|^2+ \sum_{i=0}^{k-1}\mathbf{1}_{[0, \lambda_l]}(i)m_i.
\end{eqnarray}
Hence, for $\Delta<\Delta_1^*$, $\mathbb{E}[|z_{k\wedge\lambda_l}|^2]\leq\mathbb{E}|z_0|^2=\mathbb{E}|y_0-\theta\Delta f(y_0)|^2=:\bar{K}$.
Note that
\begin{equation}\label{c2sdea}
\mathbb{E}[|z_{\lambda_l}|^2\mathbf{1}_{\{\lambda_l<k\}}]\leq\mathbb{E}[ |z_{k\wedge\lambda_l}|^2]\leq \bar{K}.
\end{equation}
We now claim $|z_{\lambda_l}|>l$. Otherwise, $|z_{\lambda_l}|\leq l$. By the definition of $\lambda_l$, $|y_{\lambda_l}|>l$ and $|y_{\lambda_l-i}|\leq l$ for $i>0$. Then by $z_k=y_k-\theta \Delta f(y_k)$,
\begin{eqnarray*}
|z_{\lambda_l}|^2&=& |y_{\lambda_l}|^2-2\theta\Delta y_{\lambda_l}^Tf(y_{\lambda_l})+|f(y_{\lambda_l})|^2\Delta^2\theta^2\\
&\geq&(1+\gamma\theta\Delta)|y_{\lambda_l}|^2\\
&>&(1+\gamma\theta\Delta)l^2\geq l^2,
\end{eqnarray*}
which is a contradiction.
By \eqref{c2sdea}, for any $k>0$
\begin{equation}\label{c2Linf}
  \mathbb{P}\{\lambda_l<k\}\leq \frac{\bar{K}}{l^2}\rightarrow 0, \ \text{µ±} \ l\rightarrow\infty,
\end{equation}
that is, as $l\rightarrow\infty$, $\lambda_l\uparrow\infty$ a.s. Define $\mu_{\Delta}= \gamma-(1-2\theta)K\Delta-0.5\sigma\Delta$. For $\Delta<\Delta_1^*$, $\mu_{\Delta}>0$. Using the linear growth condition \eqref{linear growth} gives
\begin{equation*}
|z_k|^2=|y_k-\theta f(y_k)\Delta|^2\leq (1+\theta\Delta\sqrt{K})^2|y_k|^2.
\end{equation*}
Substituting this into \eqref{maina} yields that for  $\Delta<\Delta_1^*$
\begin{eqnarray}\label{mainrt}
 \nonumber  |z_{k+1}|^2 \leq \Big(1-\frac{\mu_{\Delta}\Delta}{(1+\theta\Delta\sqrt{K})^2}\Big)  |z_k|^2+m_k,
\end{eqnarray}
which implies
\begin{eqnarray*}
 e^{\gamma_{\Delta}k\Delta} |z_k|^2\leq |z_0|^2 +\sum_{j=0}^{k-1}e^{\gamma_{\Delta}j\Delta}m_j.
\end{eqnarray*}
where $\gamma_{\Delta}=-\frac{1}{\Delta}\log\Big(1-\frac{\mu_{\Delta}\Delta}{(1+\theta\Delta\sqrt{K})^2}\Big)$. Replacing $k$ by $k\wedge\lambda_l$ yields
$$
 e^{\gamma_{\Delta}(k\wedge\lambda_l)\Delta} |z_{k\wedge\lambda_l}|^2\leq|z_0|^2 +\sum_{j=0}^{k-1}e^{\gamma_{\Delta}j\Delta}\mathbf{1}_{[0,\lambda_l]}(j)m_j.
$$
Taking expectation gives
 \begin{eqnarray*}
  \mathbb{E}[e^{\gamma_{\Delta}(k\wedge\lambda_l)\Delta} |z_{k\wedge\lambda_l}|^2]\leq \mathbb{E}|z_0|^2.
\end{eqnarray*}
By Fatou's Lemma, letting $l\rightarrow\infty$, we have
 \begin{eqnarray}\label{c200}
e^{\gamma_{\Delta}k\Delta} \mathbb{E}[|z_{k}|^2]\leq \mathbb{E}|z_0|^2,
\end{eqnarray}
which together with \eqref{z>yk} implies the desired assertion.

\textbf{Case (ii)}\ $\theta\in(1/2, 1]$:
Since $\theta\in(1/2,1]$, by $\Delta f(y_k)=(y_k-z_k)/\theta$, we obtain from the definition of $ m_k$
\begin{eqnarray*}
\nonumber m_k&=&2\Big\langle y_k+\frac{1-\theta}{\theta}(y_k-z_k),g(y_k)\Delta w_k+\frac{1}{2}L^1g(y_k)(|\Delta w_k|^2-\Delta)\Big\rangle\\
\nonumber&&+|g(y_k)|^2(|\Delta w_k|^2-\Delta)+\langle g(y_k)\Delta w_k,L^1g(y_k)(|\Delta w_k|^2-\Delta)\rangle\\
&&+\frac{1}{4}|L^1g(y_k)|^2[(|\Delta w_k|^2-\Delta)^2-2\Delta^2]
\end{eqnarray*}
and from \eqref{coupled condition}
\begin{equation}
|g(y_k)|^2 \leq -2y_k^Tf(y_k)=-\frac{2}{\Delta\theta}[y_k^T(y_k-z_k)]\leq\frac{1}{\theta\Delta}(3|y_k|^2+|z_k|^2).
\end{equation}
This, together with \eqref{lgcondition}, produces that there is a positive constants $\underline{C}$ such that
\begin{equation}\label{524c}
|m_k|\leq \underline{C}[1+|\Delta w_k|^2+|\Delta w_k|^4]|y_k|^2+\underline{C}[1+ |\Delta w_k|^2]|z_k|^2.
\end{equation}
We therefore have $\mathbb{E}[ m_k \mathbf{1}_{[0,\lambda_l]}(k)]=0$.  By \eqref{main2} and  $\Delta f(y_k)=(y_k-z_k)/\theta$,
\begin{eqnarray}\label{mainooo}
|z_{k+1}|^2\leq\frac{(1-\theta)^2}{\theta^2}|z_k|^2+\Big[\frac{1-2\theta}{\theta^2}-(\gamma-0.5\sigma\Delta)\Delta\Big]|y_k|^2+\frac{4\theta-2}{\theta^2} z_k^Ty_k+m_k.
\end{eqnarray}
Note that $\Delta<\frac{2\gamma}{\sigma}$ implies $2\theta-1+(\gamma-0.5\sigma\Delta)\Delta\theta^2>0$. Since
$$
2z_k^Ty_k\leq \frac{2\theta-1+(\gamma-0.5\sigma\Delta)\Delta\theta^2}{2\theta-1}|y_k|^2+\frac{2\theta-1}{ 2\theta-1+(\gamma-0.5\sigma\Delta)\Delta\theta^2}|z_k|^2,
$$
we get from \eqref{mainooo}
\begin{eqnarray}\label{mainbc}
 \nonumber|z_{k+1}|^2 &\leq& \Big[\frac{(1-\theta)^2}{\theta^2}+\frac{(2\theta-1)^2}{\theta^2( 2\theta-1+(\gamma-0.5\sigma\Delta)\Delta\theta^2)}\Big]|z_k|^2+m_k\\
 &=&|z_k|^2+\frac{(1-2\theta)(\gamma-0.5\sigma\Delta)\Delta }{ 2\theta-1+(\gamma-0.5\sigma\Delta)\Delta\theta^2} |z_k|^2+m_k,
\end{eqnarray}
By the similar techniques used in case of $\theta\in[0,1/2]$, we can prove that $\lambda_l\uparrow\infty$ a.s., as $l\rightarrow\infty$. Note that \eqref{mainbc} implies
\begin{eqnarray*}
 e^{\gamma_{\Delta}k\Delta} |z_k|^2\leq |z_0|^2 +\sum_{j=0}^{k-1}e^{\gamma_{\Delta}j\Delta}m_j,
\end{eqnarray*}
where $\gamma_{\Delta}=-\frac{1}{\Delta}\log\Big(1-\frac{(2\theta-1)(\gamma-0.5\sigma\Delta)\Delta}{2\theta-1+(\gamma-0.5\sigma\Delta)\Delta\theta^2}\Big).$ Then we obtain the desired assertion by repeating the proof process of the case (i).
\end{proof}

\begin{remark}
{\rm  Theorem \ref{thSSST} shows that the two theta-Milstein schemes can share the exponential mean-square stability of the exact solution. In fact, the upper bound $\gamma$ of the decay rate can also be reproduced arbitrarily accurately for sufficiently small stepsize $\Delta$ since $\lim_{\Delta\rightarrow0}\gamma_{\Delta}=\gamma$.}
\end{remark}

Let us now examine the linear scalar system
\begin{equation}\label{linear SDE}
  dx(t)=\mu x(t)dt+cx(t)dw(t)
\end{equation}
It is known that mean-square stability for \eqref{linear SDE} is equivalent to
\begin{equation}\label{linear msstabilty}
2\mu+c^2<0.
\end{equation}
Moreover, it is easy to observe that the two classes of Milstein schemes are equivalents for the linear SDE \eqref{linear SDE}. Here we rewrite this scheme as follows:
\begin{equation}\label{linearMilstein}
  y_{k+1}=y_k+\theta\mu\Delta y_{k+1}+(1-\theta)\mu\Delta y_k+c y_k\Delta w_k+\frac{c^2}{2} y_k[|\Delta w_k|^2-\Delta].
\end{equation}

For the linear scalar SDE \eqref{linear SDE}, we have the following stability theorem.
\begin{theorem}\label{thlinear}
   Let condition \eqref{linear msstabilty} hold. If
   \begin{equation}\label{linearmuc}
    (2\mu+c^2)+\frac{1}{2}c^4\Delta+(1-2\theta)\mu^2\Delta<0,
   \end{equation}
    then the Milstein scheme \eqref{linearMilstein} holds
    \begin{equation*}
    \mathbb{E}|y_k|^2=e^{-\gamma_{\Delta}k\Delta} \mathbb{E}|x_0|^2,
    \end{equation*}
    where $\displaystyle\gamma_{\Delta}=-\frac{1}{\Delta}\log\Big(1+\frac{[2\mu+c^2+(1-2\theta)\mu^2\Delta+1/2c^4\Delta]\Delta}{(1-\theta\Delta\mu)^2}\Big).$
\end{theorem}
\begin{proof}
Rearranging equation \eqref{linearMilstein} gives
$$
y_{k+1}=\frac{1}{1-\theta\mu\Delta}\Big(1+(1-\theta)\mu\Delta+c\Delta w_k+\frac{c^2}{2}[|\Delta w_k|^2-\Delta]\Big) y_k.
$$
Note that $y_k$ is $\mathfrak{F}_{k\Delta}$- measurable and $\Delta w_k$ is independent of $\mathfrak{F}_{k\Delta}$. We therefore have
   \begin{eqnarray*}
       \mathbb{E}|y_{k+1}|^2&=& \frac{[1+(1-\theta)\mu\Delta]^2+c^2\Delta+\frac{c^4}{2}\Delta^2}{(1-\theta\mu\Delta)^2} \mathbb{E}|y_k|^2\\
     &=&\Big[1+\frac{[2\mu+c^2+(1-2\theta)\mu^2\Delta+1/2c^4\Delta]\Delta}{(1-\theta\mu\Delta)^2}\Big] \mathbb{E}|y_k|^2\\
     &=&e^{-\gamma_{\Delta}(k+1)\Delta}\mathbb{E}|x_0|^2,
   \end{eqnarray*}
   where we used $\mathbb{E}\Delta w_k=0$, $\mathbb{E}|\Delta w_k|^2=\Delta$, $\mathbb{E}(\Delta w_k)^3=0$ and $\mathbb{E}|\Delta w_k|^4=3\Delta^2$.
\end{proof}
\begin{remark}
{\rm Theorem \ref{thlinear} shows that $(i)$ for $\theta\in[0,1/2]$, if $\Delta<\Delta^*:=\frac{-2\mu-c^2}{0.5c^4+(1-2\theta)\mu^2}$, Milstein scheme \eqref{linearMilstein} shares the exponential mean-square stability of the exact solution. $(ii)$ For $\theta\in(1/2,1]$, if $\mu^2<\frac{c^4}{2(2\theta-1)}$ (the diffusion term plays a crucial role), then for $\Delta<\Delta^*$, Milstein scheme \eqref{linearMilstein} is exponentially mean-square stable, and if $\mu^2\geq\frac{c^4}{2(2\theta-1)}$ (the drift term plays a crucial role), then Milstein scheme \eqref{linearMilstein} is exponentially mean-square stable unconditionally. These results are coincident with Theorem 2.1 in \cite{H2000b}. Theorem \ref{thlinear} also presents the exponential decay rate $\gamma_{\Delta}$ of theta-Milstein scheme.}
\end{remark}

\end{document}